\documentclass[reqno]{amsart}

\usepackage{amsmath,amsfonts,amssymb,amscd,verbatim,delarray,verbatim,epigraph}

\usepackage{amsthm}

\usepackage{url}

\newcommand{\F}{{\mathbb F}}

\newcommand{\Q}{{\mathbb Q}}

 \DeclareMathOperator{\tr}{tr}

\begin{document}

\newtheorem{theorem}{Theorem}

\newtheorem{corollary}{Corollary}
\newtheorem{corol}{Corollary}
\newtheorem{conj}{Conjecture}

\theoremstyle{definition}
\newtheorem{defn}{Definition}
\newtheorem{example}{Example}

\newtheorem{remarks}{Remarks}
\newtheorem{remark}[remarks]{Remark}

\newtheorem{question}{Question}
\newtheorem{problem}{Problem}

\newtheorem{quest}{Question}

\def\toeq{{\stackrel{\sim}{\longrightarrow}}}
\def\into{{\hookrightarrow}}


\def\alp{{\alpha}}  \def\bet{{\beta}} \def\gam{{\gamma}}
 \def\del{{\delta}}
\def\eps{{\varepsilon}}
\def\kap{{\kappa}}                   \def\Chi{\text{X}}
\def\lam{{\lambda}}
 \def\sig{{\sigma}}  \def\vphi{{\varphi}} \def\om{{\omega}}
\def\Gam{{\Gamma}}   \def\Del{{\Delta}}
\def\Sig{{\Sigma}}   \def\Om{{\Omega}}
\def\ups{{\upsilon}}


\def\F{{\mathbb{F}}}
\def\BF{{\mathbb{F}}}
\def\BN{{\mathbb{N}}}
\def\Q{{\mathbb{Q}}}
\def\Ql{{\overline{\Q }_{\ell }}}
\def\CC{{\mathbb{C}}}
\def\R{{\mathbb R}}
\def\V{{\mathbf V}}
\def\D{{\mathbf D}}
\def\BZ{{\mathbb Z}}
\def\K{{\mathbf K}}
\def\XX{\mathbf{X}^*}
\def\xx{\mathbf{X}_*}

\def\AA{\Bbb A}
\def\BA{\mathbb A}
\def\HH{\mathbb H}
\def\PP{\Bbb P}

\def\Gm{{{\mathbb G}_{\textrm{m}}}}
\def\Gmk{{{\mathbb G}_{\textrm m,k}}}
\def\GmL{{\mathbb G_{{\textrm m},L}}}
\def\Ga{{{\mathbb G}_a}}

\def\Fb{{\overline{\F }}}
\def\Kb{{\overline K}}
\def\Yb{{\overline Y}}
\def\Xb{{\overline X}}
\def\Tb{{\overline T}}
\def\Bb{{\overline B}}
\def\Gb{{\bar{G}}}
\def\Ub{{\overline U}}
\def\Vb{{\overline V}}
\def\Hb{{\bar{H}}}
\def\kb{{\bar{k}}}

\def\Th{{\hat T}}
\def\Bh{{\hat B}}
\def\Gh{{\hat G}}


\def\cC{{\mathcal C}}
\def\cU{{\mathcal U}}
\def\cP{{\mathcal P}}
\def\cV{{\mathcal V}}
\def\cS{{\mathcal S}}

\def\CG{\mathcal{G}}

\def\cF{{\mathcal {F}}}

\def\Xt{{\widetilde X}}
\def\Gt{{\widetilde G}}


\def\hh{{\mathfrak h}}
\def\lie{\mathfrak a}

\def\XX{\mathfrak X}
\def\RR{\mathfrak R}
\def\NN{\mathfrak N}

\def\minus{^{-1}}

\def\GL{\textrm{GL}}            \def\Stab{\textrm{Stab}}
\def\Gal{\textrm{Gal}}          \def\Aut{\textrm{Aut\,}}
\def\Lie{\textrm{Lie\,}}        \def\Ext{\textrm{Ext}}
\def\PSL{\textrm{PSL}}          \def\SL{\textrm{SL}} \def\SU{\textrm{SU}}
\def\loc{\textrm{loc}}
\def\coker{\textrm{coker\,}}    \def\Hom{\textrm{Hom}}
\def\im{\textrm{im\,}}           \def\int{\textrm{int}}
\def\inv{\textrm{inv}}           \def\can{\textrm{can}}
\def\id{\textrm{id}}              \def\Char{\textrm{char}}
\def\Cl{\textrm{Cl}}
\def\Sz{\textrm{Sz}}
\def\ad{\textrm{ad\,}}
\def\SU{\textrm{SU}}
\def\Sp{\textrm{Sp}}
\def\PSL{\textrm{PSL}}
\def\PSU{\textrm{PSU}}
\def\rk{\textrm{rk}}
\def\PGL{\textrm{PGL}}
\def\Ker{\textrm{Ker}}
\def\Ob{\textrm{Ob}}
\def\Var{\textrm{Var}}
\def\poSet{\textrm{poSet}}
\def\Al{\textrm{Al}}
\def\Int{\textrm{Int}}
\def\Smg{\textrm{Smg}}
\def\ISmg{\textrm{ISmg}}
\def\Ass{\textrm{Ass}}
\def\Grp{\textrm{Grp}}
\def\Com{\textrm{Com}}
\def\rank{\textrm{rank}}

\def\char{\textrm{char}}

\def\wid{\textrm{wd}}

\newcommand{\Or}{\operatorname{O}}

\def\tors{_\def{\textrm{tors}}}      \def\tor{^{\textrm{tor}}}
\def\red{^{\textrm{red}}}         \def\nt{^{\textrm{ssu}}}

\def\sss{^{\textrm{ss}}}          \def\uu{^{\textrm{u}}}
\def\mm{^{\textrm{m}}}
\def\tm{^\times}                  \def\mult{^{\textrm{mult}}}

\def\uss{^{\textrm{ssu}}}         \def\ssu{^{\textrm{ssu}}}
\def\comp{_{\textrm{c}}}
\def\ab{_{\textrm{ab}}}

\def\et{_{\textrm{\'et}}}
\def\nr{_{\textrm{nr}}}

\def\nil{_{\textrm{nil}}}
\def\sol{_{\textrm{sol}}}
\def\End{\textrm{End\,}}

\def\til{\;\widetilde{}\;}

\def\min{{}^{-1}}

\def\AGL{{\mathbb G\mathbb L}}
\def\ASL{{\mathbb S\mathbb L}}
\def\ASU{{\mathbb S\mathbb U}}
\def\AU{{\mathbb U}}


\title[Equations in simple groups and algebras] {Word equations in simple groups and
\\ polynomial equations in simple algebras}

\author[Kanel-Belov, Kunyavski\u\i , Plotkin]{Alexey Kanel-Belov,
Boris Kunyavski\u\i , Eugene Plotkin}

\address{Kanel-Belov: Department of
Mathematics, Bar-Ilan University, 5290002 Ramat Gan, ISRAEL}
\email{beloval@macs.biu.ac.il}

\address{Kunyavskii: Department of
Mathematics, Bar-Ilan University, 5290002 Ramat Gan, ISRAEL}
\email{kunyav@macs.biu.ac.il}

\address{Plotkin: Department of
Mathematics, Bar-Ilan University, 5290002 Ramat Gan, ISRAEL}
\email{plotkin@macs.biu.ac.il}

\begin{abstract}
We give a brief survey of recent results on word maps on simple
groups and polynomial maps on simple associative and Lie algebras.
Our focus is on parallelism between these theories, allowing one to
state many new open problems and giving new ways for solving older
ones.
\end{abstract}

\dedicatory{To Kolya Vavilov, friend and colleague, on the occasion
of the 60th anniversary}

\maketitle

\epigraph{There were different times: a time to throw stones, a time
to divide and subtract. Now it is a time to add and multiply.
Circumstances force us to focus on adding and multiplying.}{\it
{From an interview of Mikhail Silin, first vice-rector on research
of Gubkin Russian State University of Oil and Gas to the newspaper
``Vestnik of Murmansk'', December 4, 2009}}


\medskip

\section{Introduction} \label{intro}
In this paper we discuss word maps
\begin{equation} \label{wordmap}
w\colon G^d\to G,
\end{equation}
induced on any group $G$ by a group word $w=w(x_1,\dots,x_d)$ in
$x_1,x_1^{-1},\dots ,x_d,x_d^{-1}$ (= an element of the free group
$\cF_d$). We also consider polynomial maps
\begin{equation} \label{polmap}
P\colon A^d\to A,
\end{equation}
induced on any associative (or Lie) algebra over a field $k$ (or a
ring $R$) by an associative (or Lie) polynomial $P=P(X_1,\dots,X_d)$
in $d$ variables (= an element of the free associative (or Lie)
algebra).

Both \eqref{wordmap} and \eqref{polmap} are evaluation maps: one
substitutes $d$-tuples of elements of the group $G$ (algebra $A$)
instead of the variables and computes the value by performing all
group (algebra) operations. We are interested in surjectivity of
maps (\ref{wordmap}) and (\ref{polmap}), or, more generally, in
description of their images. In lowbrow terms, we are interested in
solvability of equations of the form

\begin{equation} \label{word}
w(x_1,\dots ,x_d)=g,
\end{equation}
or
\begin{equation} \label{pol}
P(X_1,\dots ,X_d)=M,
\end{equation}

\noindent for every right-hand side, or for some ``typical''
right-hand side, or whether every element of the group (algebra)
admits a representation as a product (sum) of finite number values
of the word (polynomial) map, etc.

This setting is a particular case of the following one. Let $\Theta$
be a variety of algebras, $H$ be an algebra in $\Theta$, $W(X)$ be a
free finitely generated algebra in $\Theta$ with generators $x_1,
\dots, x_d$. Fix $w\in W(X)$ and consider the word map $w\colon H^d
\to H.$ Varying $\Theta$ and $H\in\Theta$, we arrive at the problem
of solvability of equations in different varieties over different
algebras. In this note we restrict ourselves to the varieties of
groups, associative algebras, and Lie algebras.

Whereas in the group case the theory has been intensely developing
and led to several spectacular results, including answering some
old-standing questions, on the ring-theoretic side much less is
known, though some new approaches to not less old questions have
been recently found. Our main goal in this survey, which does not
pretend to be comprehensive, is to emphasize parallels between the
two theories. We hope that this may bring cross-fertilization
effects in near future. With an eye towards such a unification, we
put here more questions than answers.

The interested reader is referred to the monograph \cite{Se} and
surveys \cite{Sh1}, \cite{Ni}, \cite{BGK} in what concerns the group
case. Some references on the algebra case can be found in
\cite{KBMR}, \cite{BGKP}. We leave aside extremely interesting
questions on the number of solutions of equations \eqref{word},
\eqref{pol} (or, in other words, on the structure of the fibres of
maps \eqref{wordmap}, \eqref{polmap}). See \cite{BGK} for an
overview of some results in this direction.

\section{Value sets}

Given a word map \eqref{wordmap}, it is important to distinguish
between the following objects:
\begin{itemize}
\item the value set in strict sense: $w(G)=\{g\in G: \exists
(g_1,\dots ,g_d) \quad w(g_1,\dots ,g_d)=g\}$;
\item the symmetrized value set $w(G)^{\pm}$ consisting of the
elements of $w(G)$ and their inverses;
\item the verbal subgroup $\left<w(G)\right>$ of $G$ generated by
$w(G)$.
\end{itemize}
Respectively, given a polynomial map \eqref{polmap} of $k$-algebras,
we distinguish between the value set $P(A)=\{a\in A: \exists
(a_1,\dots ,a_d) \quad P(a_1,\dots ,a_d)=a\}$ and the vector space
$\left<P(A)\right>$ spanned by $P(A)$ over $k$.

It is usually much easier to describe $\left<w(G)\right>$ and
$\left<P(A)\right>$ than the actual value sets.

\subsection{Surjectivity}

We start with the group case and ask whether word map
\eqref{wordmap} is surjective. More precisely,
\begin{itemize}
\item given a class of groups $\mathcal G$, we want to describe words $w$ for
which map \eqref{wordmap} is surjective;
\item given a class of words $\mathcal W$, we want to describe groups $G$ for
which map \eqref{wordmap} is surjective.
\end{itemize}

In each of these set-ups arising problems range from very easy to
extremely difficult, depending on the choice of the class. Here are
some examples. Let us start with the first approach and take
$\mathcal G$ to be the variety of all groups. Then the needed
description is given by \cite[Lemma~3.1.1]{Se}, where such words are
called universal: they are all of the form $w=x^{e_1}\dots
x_d^{e_d}w'$ where $w'$ is a product of commutators and
gcd$(e_1,\dots,e_d)=1$. Here is a parallel question for associative
algebras:

\begin{quest}
What are polynomials $P$ such that the map $P\colon A^d\to A$ is
surjective for all associative algebras $A$?
\end{quest}

When we restrict the class $\mathcal G$, we arrive at more
interesting and difficult questions. Answers heavily depend on this
choice. For instance, a theorem of Rhemtulla (see
\cite[Theorem~3.1.2]{Se}) says that if $\mathcal G$ consists of the
free groups (and maybe also of some free products adjoined), then
the behaviour of {\it any} non-universal word is very far from
surjectivity: in the terminology explained below, any such word is
of infinite width for any group of $\mathcal G$. That is why in this
survey we prefer to stay away from equations in free (and close to
free) groups; see, e.g., \cite{CRK}, particularly the introduction,
for a comprehensive bibliographical survey of vast literature on
this fascinating theory.

We mainly focus on another extreme case of {\it simple} groups which
can also be viewed as a building block for some more general theory.
Here is our first question.

\begin{quest} \label{sur-adj}
Let $\mathcal G$ be the class of simple groups $G$ of the form
$G=\mathbb G(k)$ where $k=\bar k$ is an algebraically closed field
and $\mathbb G$ is a semisimple adjoint linear algebraic group. Is
it true that for all non-power words $w\ne 1$ word map \eqref{wordmap} is
surjective?
\end{quest}

\begin{remarks}
\begin{itemize}
\item[(i)] If we drop the assumption $k=\bar k$, or allow $G$ to be not
of adjoint type (i.e., consider quasisimple groups), there are easy
counter-examples to surjectivity  \cite{My}, \cite{Bo}. On the other
hand, if $G=SU(n)$ and $w(x,y)$ is any word not belonging to the
second derived subgroup of $\cF_2$, then the induced word map is
surjective for infinitely many $n$ \cite{ET}. However, if $w$ does
belong to the second derived subgroup of $\cF_2$, it may be far from
surjective (its image can be arbitrary small in the real topology of
$G$) \cite{Th}).

\item[(ii)] If $\mathcal G$ is some infinite family of finite simple groups,
then any power word $w=x^n$ gives rise to the word map which is not
surjective for infinitely many groups (those whose order is not
prime to $n$). A conjecture of Shalev, asserting that such
phenomenon may arise only for power maps, turned out to be
over-optimistic, see \cite{JLO} for a counter-example. For simple
algebraic groups from Question \ref{sur-adj}, a power word may not
be surjective  either. See \cite{Stb}, \cite{Cha1}--\cite{Cha4} for
details.

\item[(iii)] Going over to analogues of Question \ref{sur-adj}
for finite-dimensional simple associative and Lie algebras over an
algebraically closed field, one immediately gets obvious
obstructions: there are nontrivial (associative and Lie) polynomials
that vanish identically; in addition, an associative polynomial map
of matrix algebras may take only central or trace-zero values; even
apart from these obvious obstructions, there are counter-examples to
surjectivity, see \cite{BGKP} and \cite{KBMR}; if the ground field
is not algebraically closed, the situation is even more complicated
(for example, it is an interesting question which of the
surjectivity results of \cite{KBMR}, obtained over quadratically
closed fields, survive over $\mathbb R$). Nevertheless, some of
techniques developed for associative polynomials may turn out to be
useful for attacking Question \ref{sur-adj}, see Remarks
\ref{rem-am} and \ref{rem-uni} below.
\end{itemize}
\end{remarks}

The opposite direction for studying surjectivity is more tricky,
even if the class of words $\mathcal W$ consists of only one word.
The first non-universal word to be considered is commutator.

\subsubsection{Commutator: Ore and beyond}

Let $w=xyx^{-1}y^{-1}\in \cF_2$. The obvious necessary condition for
the surjectivity of map \eqref{wordmap} is $\left<w(G)\right>=G$,
i.e., $G$ must be perfect. It is very far from being sufficient,
even within the class of finite groups, see, e.g., \cite{Is}. Making
the further assumption that $G$ is simple, one can say something
much more positive, namely that map \eqref{wordmap} is surjective in
each of the following cases:
\begin{itemize}
\item[(i)] $G$ is finite (Ore's problem, whose long history has been
finished in \cite{LOST1});
\item[(ii)] $G=S_{\infty}$, the infinite symmetric group \cite{Or},
\cite{DR} (more generally, {\it any} non-power word is surjective on
$S_{\infty}$ \cite{Ly});
\item[(iii)] $G=\mathbb G(k)$, where $\mathbb G$ is a semisimple adjoint linear algebraic group
over an algebraically closed field $k$ \cite{Ree} (and for many
other similar cases, such as semisimple Lie groups, etc., see
\cite{VW});
\item[(iv)] $G$ is the automorphism group of some ``nice''
topological or combinatorial object (e.g., the Cantor set), as
listed in \cite{DR}; see also \cite{BM}.
\end{itemize}

In connection with (iv), it is worth quoting I.~Rivin (see
\url{http://mathoverflow.} \newline
\url{net/questions/77398/how-did-ores-conjecture-become-a-conjecture}):
``It is a conjecture I attribute to myself, but probably goes back
to the ancients, that in every reasonable simple group every element
is a commutator''. Of course, the whole point here is in the word
``reasonable'': as any meaningful principle, Rivin's principle is
subject to breaking, see Section \ref{monster} for counter-examples.
In a positive direction, we suggest the following question:

\begin{quest}
For each of the infinite simple groups listed in \cite{DR}, is it
true that any word map is surjective (or, at least, has a ``large''
image)? (In other words, is it true that in each such group one can
solve word equations with ``generic'' right-hand side?)
\end{quest}

We also note that the Ore property is not just a standing alone
phenomenon in group theory. It has many interesting applications in
algebraic topology (see \cite{DR} and references therein) and
birational geometry \cite{Ku}.

\begin{remarks} \label{beyondOre}
There are several ways to go beyond Ore.
\begin{itemize}
\item[(i)] First, representing $1\ne g\in G$ in the form
$g=[x,y]$, one may require to make the choice of $x$ as ``uniform''
as possible. A typical result of such kind is a theorem of Gow
\cite{Gow}: every finite simple group $G$ of Lie type contains a
regular semisimple conjugacy class $C$ such that each semisimple
$g\in G$ can be represented in the form $g=[x,y]$ with $x\in C$. A
similar result in the case where $G$ is a split semisimple Chevalley
group over an infinite field follows from the prescribed Gauss
decomposition \cite{EG}. The case where $G$ is not split seems open.

\item[(ii)] One can require that $x,y$ satisfy some additional
properties. For example, in \cite{DR} Ore's theorem on the infinite
symmetric group has been strengthened by choosing $x$ and $y$ so
that they generate a transitive subgroup of $S_{\infty}$. It would
be interesting to get similar results in the other cases listed
above. For instance, one can ask how big can the subgroup
$\left<x,y\right>\subset G$ be. Note that we cannot guarantee
$\left<x,y\right>=G$, there are counter-examples among alternating
groups (we thank A.~Shalev for this remark).

\item[(iii)] Another well-known generalization is Thompson's
conjecture asserting the existence of a conjugacy class $C$ in every
simple group $G$ such that $C^2=G$. This problem has a positive
solution  for the split semisimple algebraic groups \cite{EG}.
However, it is still open for finite simple groups, in spite of the
major breakthrough made in \cite{EG} (see also \cite{YCW} and some
more recent results (see, e.g., \cite{BGK} for relevant
references)). The case of the infinite symmetric group has been
settled in \cite{Gr}. It is a natural question whether there is a
``gap'' between Ore's and Thompson's conjectures, i.e., whether
there is a simple group where Ore's conjecture holds and Thompson's
does not. N.~Gordeev pointed out that such a group exists: take
$G=A_{\infty}$, the finitary alternating group. Then evidently each
element of $G$ is a commutator but in every fixed conjugacy class
$C$ of $G$ any element moves a fixed number $N$ of points, hence any
product of two elements of $C$ moves at most $2N$ points. Therefore
$C^2\ne G$.

\item[(iv)] In the case where $L$ is a ``classical'' Lie algebra (i.e.,
the Lie algebra of a Chevalley group $G$ over a sufficiently large
field), most of the statements discussed above admit easy analogues
when the bracket stands for the Lie operation. Say, the analogue of
Ore's conjecture has been established in \cite{Br}, and an analogue
of Gow's theorem can easily be obtained looking at the linear map
$\varphi_x\colon L\to L$, $y\mapsto [x,y]$: its kernel is the
centralizer of $x$, and one can conclude that in the representation
$g=[x,y]$ the element $x$ can be chosen from a fixed $G$-orbit (with
respect to the adjoint action of $G$ on $L$). Further, for the Lie
algebras of such kind the analogue of Thompson's conjecture also
holds true \cite{Gorde}. However, all such generalizations seem
unknown for non-split simple classical Lie algebras, simple Lie
algebras of Cartan type in positive characteristic, and for
infinite-dimensional Lie algebras. The first case to be explored is
that of Kac--Moody algebras (as well as Kac--Moody groups, where the
prescribed Gauss decomposition is known \cite{MP}).

\item[(v)] Let now $A$ be an associative algebra, and the
polynomial under consideration be the additive commutator
$P(X,Y)=XY-YX$. Let $[A,A]$ denote the subalgebra additively
generated by the values of $P$. If $A\ne [A,A]$, there is no hope
for the surjectivity of the map $P\colon A^2\to A$. This happens
whenever $A$ is finitely generated over $\mathbb Q$ or, more
generally, if $A$ is any (not necessarily finitely generated)
PI-algebra \cite{Be}, \cite{Mes}. So reasonably interesting
questions only arise for the kernel of the reduced trace map (trace
zero matrices, for brevity). If $A=M_n(D)$, where $D$ is a division
algebra over a field and $n\ge 2$, then every trace zero matrix is
an additive commutator \cite{AR}. If, however, $n=1$, the question
is open (except for the special case of a central division algebra
of prime degree over a local field \cite{Ros}), and it is hard to
believe in an easy affirmative answer (that would imply the
affirmative solution of an old open problem on the cyclicity of any
central simple algebra of prime degree $p$ over a field of
characteristic $p$, see \cite{Ros} for details). From discussions
with L.~Rowen we got an impression that one should rather expect a
negative answer, and a counter-example could arise in the algebra of
generic matrices (see below). This shows that the questions posed in
(iv) for non-split simple Lie algebras are probably answered in the
negative. Note that a too straightforward attempt to extend the
theorem of Amitsur--Rowen to $M_n(R)$, where $R$ is an arbitrary
ring, fails in general \cite{RR2}, \cite{Ros}, \cite{Mes}. As to
infinite-dimensional simple algebras, the questions seem to be
totally unexplored, to the best of our knowledge.

\end{itemize}
\end{remarks}

\subsubsection{Engel words}
Apart from the commutator, there are some other interesting words
for which the surjectivity is known in several cases; see, e.g.,
\cite{BGK} for some relevant references. Here we only want to
mention the most natural generalization of the commutator, the
family of Engel words $e_n(x,y)=[[[x,y],y]\dots y]$. Their
surjectivity has been only established for groups $G=\PSL(2,q)$
(under some assumptions on $n$ and $q$) \cite{BGG}, $G=\SU(n)$
\cite{ET}, $G=\PSL_2(\CC)$ \cite{Kl}, \cite{BGG} and for simple
classical Lie algebras \cite{BGKP}. An answer to the following
question would be a natural continuation of \cite{Ree}.

\begin{quest} \label{engel}
Let $G$ be a connected semisimple adjoint linear algebraic group
over an algebraically closed field. Is it true that every Engel word
induces a surjective map $G^2\to G$?
\end{quest}

\subsection{Dominance: the Deligne--Sullivan trick and Amitsur's theorem}
We start with a theorem of Borel \cite{Bo} providing a sketch of a
somewhat new approach (the proof given in \cite{La} is essentially
the same as in the original paper). Our proof is based on using the
generic division algebra (see \cite{Fo2} for details on this
important object, including the history of its creation). Let $F$ be
a field (for simplicity, assumed of characteristic zero), let $n$ be
a positive integer, let $x^k_{ij}$ $(1\le i,j\le n$, $k\in \mathbb
N$) be independent commuting indeterminates. The $F$-subalgebra of
$M_n(F[x^k_{ij}])$ generated by the matrices $X_k=(x^k_{ij})$ is
called a ring of generic matrices. Denote it by $R=F\{X\}$. It is a
domain \cite{Am2}, and its ring of fractions $Q(R)$ is a central
division algebra of dimension $n^2$ over its centre $Z(Q(R))$
\cite{Am1}. The centre $Z(R)$ of $R$ consists of the central
polynomials (and is hence nontrivial for every $n$
\cite{Fo1},\cite{Ra}). The centre $C(Q(R))$ is the field of
quotients of $C(R)$. It is a long-standing open problem whether the
field $C(Q(R))$ is a purely transcendental extension of $F$, see the
surveys \cite{Fo2}, \cite{LB}, \cite{ABGV}.

\begin{theorem} \cite{Bo} \label{Borel}
Let $w\in \cF^d$ ($d\ge 2)$ be a nontrivial word, and let $G$ be a
connected semisimple algebraic group over a field $F$. Then the
induced morphism $w\colon G^d\to G$ of underlying algebraic
varieties is dominant (i.e., its image is Zariski dense).
\end{theorem}

\begin{proof}
We present a sketch of proof in the case where $F$ is of
characteristic zero. Several parts of the proof are exactly the same
as in the original one. First, as dominance is compatible with any
extension of the ground field, we may and will assume $F$ to be
algebraically closed of infinite transcendence degree. Next, we may
assume that $G$ is simple of type $A_n$ (the reduction to the simple
case and the passage from $\SL_n$ to the other types are as in
\cite{Bo}). Further, it is enough to prove the dominance for the
case where $w$ is a product of commutators. Indeed, suppose that the
theorem is proven for such words, and let $w\colon G^d\to G$ be an
arbitrary word map. Consider the map $\tilde w\colon G^{2d}\to G$
defined as follows: $\tilde w(x_1,y_1,\dots ,x_d,y_d) :=
w([x_1,y_1],\dots ,[x_d,y_d])$. Then the image of $\tilde w$ is
dense, hence so is the image of $w$. Thus henceforth we assume $w\in
[\cF_d,\cF_d]$. Furthermore, it is enough to prove the dominance for
the map
\begin{equation} \label{gl}
w\colon (\GL_n)^d\to \SL_n.
\end{equation}
Indeed, the image of this map coincides with $w(\SL_n)$ because
every $g\in\GL_n$ can be replaced with $g/\det(g)$.

Let us now argue by induction on the rank. The case of rank 1 is
treated exactly as in \cite{Bo}. The key point is the induction
step. Assume that map \eqref{gl} is dominant for the rank $n-1$ and
prove that it is dominant for the rank $n$. As in the original
proof, it is enough to prove the existence of a matrix $C$ in the
image of $w$ such that none of its eigenvalues equals 1.

To this end, it suffices to prove that the image contains a {\it
generic} matrix with this property. Indeed, as soon as this is
established, by specialization arguments (which are legitimate
because of the assumption on the transcendence degree of $F$) we
obtain a non-empty Zariski open set of needed matrices.

To prove the existence of a generic matrix with the required
property, assume the contrary. Let $C$ denote a generic matrix from
the image of $w$. Since it has an eigenvalue 1, the characteristic
polynomial $p_C(t)$ is divisible by $t-1$, so $p_C(t)=(t-1)p_1(t)$.
The Cayley--Hamilton theorem then gives $(C-I_n)p_1(C)=0$, which is
in a contradiction with Amitsur's theorem stating that the ring of
generic matrices is a domain.
\end{proof}

\begin{remark} \label{rem-am}
In the proof presented above, the core induction argument is based
on Amitsur's theorem, instead of the Deligne--Sullivan argument used
in the original proof (which is based on going over to an
anisotropic form of $G$ and dates back to the unitary trick of
Weyl). We believe that this interrelation between matrix groups and
algebras is very important, and its potential is not exhausted. In
particular, such an approach may be useful for proving the
surjectivity of the word map for semisimple groups over
algebraically closed fields and, more generally, for getting a more
precise description of the image of the word map.

Let $w=\prod_j a_{i(j)}^{\pm 1}.$ Consider the values of $w$ on
invertible matrices, which will be denoted by the letters $a_i$,
$i=1,\dots, k$. The matrix $a_i$ satisfies the Cayley--Hamilton
equation
\begin{equation}     \label{eqHmCl}
a_i^n-\xi_1(a_i)a_i^{n-1}+\dots+(-1)^l\xi_l(a_i)a_i^{n-l}+\dots+(-1)^n\xi_n(a_i)=0
\end{equation}
where $\xi_l$ is the $l^{th}$ characteristic coefficient of the
matrix $a_i$. In particular, $\xi_1=\tr$, $\xi_n=\det$. In the zero
characteristic case, the $\xi_l$ are expressed via traces of powers.
Equation (\ref{eqHmCl}) can be rewritten in the form
$$
(\sum_{l=1}^{n-1}(-1)^{n-l}a_i^{n-l}\xi_{l-1}(a_i))\det(a_i)^{-1}=a_i^{-1}.
$$
Set $\psi(a_i)=\sum_{l=1}^{n-1}(-1)^{n-l}a_i^{n-l}\xi_{l-1}(a_i)$
and rewrite the product $w=\prod_j a_{i(j)}^{\pm 1}$, replacing
$a_{i(j)}^{-1}$ with $\psi(a_{i(j)})$ and leaving the factors
$a_{i(j)}$ unchanged. We will obtain a polynomial (in the signature
enlarged by the characteristic coefficients, in other words, a
polynomial with forms), which is a product of polynomials in one
variable.

There are some optimistic considerations regarding the study of
images of such polynomials. Let us look at the construction of
homogeneous polynomials whose image does not coincide with the whole
space \cite{KBMR}. Let $P$ be an arbitrary matrix polynomial, and
let $c\in k$. Consider the product
$$
F=(\lambda_1(P)-c\lambda_2(P))(\lambda_2(P)-c\lambda_1(P))P=
((1+c^2+2c)\det(P)-c\tr(P)^2)P.
$$
It vanishes if the ratio of the eigenvalues $\lambda_i(P)$ equals
$c^{\pm 1}$ or they are both zero. Therefore the image of $F$ does
not contain nonzero matrices with such a ratio of eigenvalues, and
this construction is essential for the nature of the problem. If we
are given a nontrivial product of polynomials in one variable in
which several variables enter, then a requirement on the resulting
relation between eigenvalues cannot be determined by one factor (if
its determinant is not equal to zero). Of course, $\psi$ is obtained
by canceling the determinant, but nevertheless there is some ground
for optimism regarding making Borel's theorem more precise.
\end{remark}

\begin{remark} \label{rem-uni}
Let us give an example of a more complicated argument showing the
existence of a unipotent element in the image of a multilinear
polynomial $P$ in the space of matrices of second order (assuming it
has a noncentral value with nonzero trace).

So let $P(X_1,\dots,X_n)$ be a multilinear polynomial, and assume
the contrary, i.e., $P$ has no unipotent values. Fix all variables
except for $X_1$. Then $P_{X_2,\dots,X_n}(X_1)=P(X_1,\dots,X_n)$ is
a linear in $X_1$ function, meeting the discriminant surface at
points of total multiplicity 2. Since we assume the absence of
unipotent values, each such point is either scalar or nilpotent. If
one point, $X_1^0$, is a nonzero scalar, and the other, $X_1^1$, is
nilpotent, then at the point $X_1^0+X_1^1$ we obtain a matrix
proportional to a unipotent one with some coefficient $\lambda\ne
0$, and then at the point $(X_1^0+X_1^1)/\lambda$ we shall obtain a
unipotent value.

Further, if at both points we have a scalar value, then the original
value of $P$ is scalar, and if both values are nilpotent, then the
original value has zero trace. Therefore, if the value of
$P(X_1,\dots,X_n)$ is not scalar and has nonzero trace, then while
moving any coordinate we observe coincidence of the intersection
points with the discriminant surface, i.e., tangency to the
discriminant surface.

Then one can extract square root of the discriminant
$(\lambda_1-\lambda_2)^2$. In other words,
$q=\lambda_1(P(X_1,\dots,X_n))- \lambda_2(P(X_1,\dots,X_n))$ is a
polynomial in entries of $X_1,\dots,X_n$. The group $\SL_2$ acts by
simultaneous conjugations on the $X_i$ and is connected. Then, under
such an action, $q$ is mapped to $\pm q$, and connectedness
guarantees the invariance of $q$, by the first fundamental theorem
(established in \cite{Pr} in characteristic zero and in \cite{Do} in
positive characteristic).

Since $q$ is a polynomial in characteristic values of the products
of the $X_i$, so are $\lambda_1(P), \lambda_2(P)$ because
$\lambda_1(P)+\lambda_2(P)=\tr(P)$.

By the Cayley--Hamilton theorem,
$(P-\lambda_1(P))(P-\lambda_2(P))=0$ which also contradicts
Amitsur's theorem on zero divisors. Hence $P$ does have unipotent
values.

\medskip

We hope that such kind of reasoning can be helpful in getting a more
precise description of the image of the word map in the set-up of
Borel's theorem.

\end{remark}

\begin{remark}
Note that an idea of using generic matrices has been recently used,
in a slightly different context of ``universal localization'', in
\cite{HSVZ1} and \cite{Step}, for getting subtle information on
commutators in Chevalley groups over rings.
\end{remark}

\begin{remark}
There are results of Borel's flavour, stating that for some infinite
groups $G$ any ``generic'' element $g\in G$ falls into the value set
of any non-power word $w$ \cite{Ma}, \cite{DT}.
\end{remark}

\begin{remark}
In the case of Lie algebras, an analogue of Borel's theorem for Lie
polynomials $P$ which are not identically zero, was established in
\cite{BGKP} for semisimple Chevalley algebras (modulo several
exceptions over fields of small characteristic), under the
additional assumption that $P$ is not an identity of
$\mathfrak{sl}(2)$. It is not clear whether this assumption can be
removed. The answer heavily depends on whether one can extend a
construction of so-called $3$-central polynomials (see, e.g.,
\cite[Theorem~3.2.21]{Row}) to the Lie case. (Recall that a
polynomial $P$ is called $n$-central is $P^n$ is central but $P$ is
not.) If such polynomials exist, they provide an example of a map
which is not dominant. Probably, multilinear Lie polynomials with
such a property do not exist, and in this case one can drop the
assumption mentioned above.
\end{remark}

\begin{remark}
The case of associative algebras is much more tricky. First, one has
to take into account the obvious obstructions to dominance, and
assume that the polynomial $P$ is non-central and contains at least
one value with nonzero trace. Even under these assumptions, there
are counter-examples to dominance for maps on $M_2(K)$; to avoid
them, one has to make additional assumptions on $P$ (say, to assume
that it is semi-homogeneous). This assumption is not enough for
$3\times 3$-matrices. One can show that if $P$ is a nonscalar
polynomial such that not all its values on $A=M_3(K)$ are 3-scalar
or traceless, then its value set $P(A)$ is dense in $A$ (moreover,
if $N$ is the set of non-diagonalizable matrices, then $P(A)\cap N$
is dense in $N$). If $P(A)$ lies in $\mathfrak{sl}_3(K)$ and
contains a matrix which is not 3-scalar, then $P(A)$ is dense in
$\mathfrak{sl}_3(K)$.

If $P$ is assumed multilinear, the main open question is the
validity of the Kaplan\-sky--L${}^{\prime}$vov conjecture which
states that the image of $P$ is either 0, or $K$, or
$\mathfrak{sl}_n(K)$, or $M_n(K)$. As a first step, we would suggest
to look at a downgraded version of this conjecture where we weaken
it by allowing the image to be a dense subset of either
$\mathfrak{sl}_n(K)$, or $M_n(K)$. As above, one of the key problems
is the existence of $n$-central multilinear polynomials. This
problem is related to subtle properties of division algebras and is
not discussed in this survey. See \cite{KBMR} for some details.

To sum up, we present a list of questions which seem to be a good
starting point.
\end{remark}

\begin{quest} Let $P\colon M_2(K)^d\to M_2(K)$.
\begin{itemize}
\item[(i)] If $P$ is multilinear, what are its possible images for $K=\mathbb R$? In
particular, does the Kaplansky--L${}^{\prime}$vov conjecture hold in
this case?
\item[(ii)] Let $K$ be quadratically closed and $P$ be an arbitrary
(not necessarily homogeneous) polynomial. Is it true that its value
set is either the set $\{F([x,y])\}$ of values of a traceless
polynomial, or the whole $M_2(K)$ (up to Zariski closure)?
\end{itemize}
\end{quest}

For arbitrary $n\ge 3$ we ask weaker questions.

\begin{quest} Let $P\colon M_n(K)^d\to M_n(K)$ be a multilinear
polynomial.
\begin{itemize}
\item[(i)] Can it be $n$-central?
\item[(ii)] Suppose that $P$ is not $n$-central. Is it true that its
image is then dense provided it contains a matrix with nonzero
trace? Can it contain non-diagonalizable matrices?
\end{itemize}
\end{quest}

Similar questions can be asked for traceless polynomials.

\begin{remark}
Being even more modest, one can start with multilinear Lie
polynomials, asking similar questions. Even the case of algebras of
small rank is open. Let us mention an analogue of the
Kaplansky--L${}^{\prime}$vov conjecture.
\end{remark}

\begin{quest}
Let $L$ be a Chevalley Lie algebra, and let $P$ be a multilinear Lie
polynomial. Is it true that the image of $P$ on $L$ is either $0$,
or $L$?
\end{quest}

\begin{remark}
Even a further downgrading might be of interest, in view of eventual
generalizations of Makar-Limanov's Freiheitsatz \cite{ML} to the
case of positive characteristic: given a polynomial map $P\colon
M_n(K)^d\to M_n(K)$ such that $n\gg\deg P$, can one bound from above
the codimension of its image by a ``reasonable'' function in $n$?
\end{remark}

\subsection{Words with small image}
 If $G$ is a finite simple group, the verbal subgroup $\left<w
 (G)\right>$ equals either 1 or $G$. However, for any $G$ one can
 expect the existence of $w$ such that the actual image $w(G)$ is
 fairly small (though different from 1). Such examples were recently
 constructed in \cite{KN} and \cite{Le}.

 \begin{quest} (N. Nikolov)
 \footnote{This question was answered in the affirmative by A.~Lubotzky
 when the paper was in print. See: A.~Lubotzky, {\it Images of word maps in finite simple
 groups}, arXiv:1211.6575.}
 Let $G$ be a finite simple group, and let $X\subset G$ be a union of
 conjugacy classes. Does there exist $w$ such that $w(G)=X$?
 \end{quest}

 Note that a similar question for polynomials on matrix algebras
 over  finite rings was answered in the affirmative in \cite{Chu}.

\section{Width} \label{sec:Engel}

\subsection{General words and polynomials}
In the situation where map \eqref{wordmap} is not surjective (or
this is not known), but $\left< w(G)\right>=G$, one can ask how many
elements of $w(G)$ (or of $w(G)^{\pm}$) are needed to represent
every element of $G$. (In the case where $\left< w(G)\right>\ne G$,
one represents every element of $\left< w(G)\right>$.) The smallest
such number is called the $w$-width of $G$. We denote it by
$\wid_w(G)$. If $G$ does not have finite $w$-width, it is said to be
of infinite $w$-width. The reader can find comprehensive discussions
of this notion in \cite{Se}, \cite{Ni} (see also \cite{BGK} for a
survey of some more recent developments in the case where $G$ is a
finite simple group). We focus here, as above, on parallels with the
case of associative algebras.

First note that as soon as we establish, for some topological group
$G$, any kind of Borel's dominance theorem, in the sense that the
image of $w$ contains a dense open subset, we immediately conclude
by a standard argument that $\wid_w(G)\le 2$. Thus this is the case
if $G$ is (the group of points of) a connected semisimple algebraic
group over an algebraically closed field. For arbitrary algebraic
groups over fields, one can still prove their finite width with
respect to any nontrivial word and any element of the verbal
subgroup \cite{Mer}. If, however, we go over to algebraic groups
over arbitrary rings, the situation changes dramatically, see the
next section.

In the class of finite simple groups, dominance arguments do not
make much sense. However, the observation made in \cite{La} that
word maps still have, in a sense, a large image when viewed within
an infinite family of finite simple groups, gave rise to a series of
wonderful results on uniform word width, christened by Shalev
``Waring type properties''. Making the long story short, we just
mention the papers \cite{LS}, where the existence of such a uniform
bound was established, and \cite{LST1}, \cite{LST2}, where, as a
culmination of efforts of many people, the uniform bounds
$\wid_w(G)\le 2$ (resp. $\wid_w(G)\le 3$) have been established for
all sufficiently large finite simple (resp. quasisimple) groups and
all words $w\ne 1$. See \cite{BGK} for more references and details.
Note that quite recently Shalev with his collaborators extended many
results of this type to some simple algebraic groups over $p$-adic
integers, see \cite{Sh2} for a list of relevant questions. In the
case of polynomial maps of algebras, we propose a notion parallel to
word width in the group case.

\begin{defn}
Let $P(X_1,\dots ,X_d)$ be an associative (resp. Lie) polynomial,
let $A$ (resp. $L$) be an associative (resp. Lie) algebra over a
ring $R$, let $P\colon A^d\to A$ (resp. $P\colon L^d \to L$) denote
the induced map, and let $V_P$ be the $R$-module spanned by the
image of $P$. If there exists a positive integer $m$ such that every
element $v\in V_P$ can be represented as a sum of $m$ values of $P$,
we call the least $m$ with such property the $P$-width of the
algebra and denote it $\wid_P(A)$ (resp. $\wid_P(L)$). Otherwise, we
say that the algebra is of infinite $P$-width.
\end{defn}

We are not aware of any general results concerning this notion,
beyond those that follow from the surjectivity or dominance (see,
however, next sections for some particular cases). The following
question seems conceptually important.

\begin{quest} \label{alg-width}
Let $k$ be an infinite field. Let $\mathcal A$ denote the class of
finite-dimensional central simple $k$-algebras. Let $\mathcal P$
denote the class of associative polynomials such that none of them
is central for some $A\in\mathcal A$. Is it true that all
$A\in\mathcal A$ are of finite $P$-width for all $P\in \mathcal P$?
If yes, is it true that $\sup_{P\in\mathcal P, A\in\mathcal
A}\wid_P(A)<\infty$?
\end{quest}

This question makes sense in more restrictive set-ups, when either
some class of polynomials or some class of algebras, or both, are
fixed (and may be narrow enough, even consisting of one element, see
some examples below). Similar questions may be posed for Lie
polynomials on finite-dimensional simple Lie algebras. Finally, in
both cases (associative and Lie) one can also consider
infinite-dimensional simple algebras, finitely or infinitely
generated. All this seems completely unexplored. Below we consider a
couple of more concrete settings.

\subsection{Commutator width}
The case of the width of various groups $G$ with respect to
$w=xyx^{-1}y^{-1}$ got much attention in the literature. Without
pretending to giving a comprehensive survey (various aspects are
reflected in \cite{Se}, \cite{Ni}, \cite{HSVZ2}, \cite{BGK}), we
only present some references.
\begin{itemize}
\item As mentioned above, the finite simple groups have commutator
width 1. As to finite quasi-simple groups, it is at most 2, and all
the cases when it actually equals 2 are listed \cite{LOST2}.

\item As mentioned above, if $G$ is a Chevalley group over
an infinite field, its commutator width is at most 2, in view of
Borel's theorem. However, this breaks down for Chevalley groups over
rings, which tend to have very few commutators (the type of
behaviour called ``anti-Ore'' in \cite{HSVZ1}). In view of examples
of groups of infinite commutator width \cite{DV}, we can also call
it ``anti-Waring''. The situation improves if we go over to infinite
matrices: most Chevalley groups of ``infinite rank'' return to
Waring (if not to Ore) behaviour: it is known that their commutator
width is at most 2 \cite{HS}, \cite{DV}; similar results were
recently obtained for other groups of infinite matrices, see
\cite{GH} and references therein. This gives rise to a natural
question:

\end{itemize}

\begin{quest}
Is it true that the groups of commutator width at most 2 listed in
\cite{DV} are of finite word width with respect to any nontrivial
word $w$? Is the image of $w$ ``large'' (i.e., do we have any
analogue of Borel's theorem)?
\end{quest}

Here are some parallel results and questions for the additive
commutator $$P(X,Y)=XY-YX$$ in associative algebras.
\begin{itemize}
\item The commutator width of a matrix algebra $A=M_n(R)$ over any
ring is at most two (\cite{AR} for the case where $R$ is a division
algebra, \cite{Ros} for an arbitrary commutative ring $R$,
\cite{Mes} in general).

\item Moreover, in the representation $M=[X,Y]+[Z,T]$ one can fix
$X$ and $Z$ which are good for every $M\in A$ (in \cite{AR} only $Z$
was fixed, and this was strengthened in \cite{RR1}, \cite{Ros} and
\cite{Mes} using genericity arguments).

\end{itemize}

This gives rise to the following notion, slightly resembling the
notion of one-and-a-half-generation of simple groups \cite{Stei},
\cite{GK}.
\begin{defn} \label{fracwid-alg}
Let $A$ be an associative algebra of commutator width $m$. If there
exist $X_1,\dots ,X_s\in A$ such that every $M\in A$ can be
represented in a form
$$
M=[X_1,Y_1]+\dots + [X_s,Y_s]+[Z_{s+1},Y_{s+1}]+\dots +[Z_m,Y_m],
$$
and $s$ is maximal with this property, we say that $A$ is of
fractional commutator width $m-\frac{s}{2m}$.
\end{defn}
Thus the matrix algebras are of fractional commutator width one and
a half (the result of \cite{AR} gave only one and three-quarters).

One can define fractional commutator width for groups in a similar
fashion and ask the following question.

\begin{quest}
Which of the groups $G$ having commutator width 2 are of fractional
commutator width one and a half? one and three-quarters?
\end{quest}

One can start with Chevalley groups of infinite rank \cite{DV} and
14 finite quasisimple groups of width 2 from the list of
\cite{LOST2}.

One can define fractional width for more general words and
polynomials. We leave this to the interested reader.

\subsection{Power width}
Again, we only briefly quote some results on power width in groups.
For finite simple groups, the story which started with establishing
the existence of uniform finite width in \cite{MZ}, \cite{SW} led to
almost conclusive results in \cite{GM} where for many powers it was
proved that this width equals 2. The uniform width for general
finite groups was established in \cite{NS}. For infinite groups
(e.g., for Chevalley groups over rings), the problem is almost
unexplored. As in the commutator case, some pathologies have been
discovered (see the next section).

For matrix algebras $M_n(R)$, we only mention the pioneering paper
\cite{Va} where the problem of representing a matrix over a
commutative ring as a sum of $d^{th}$ powers was considered, and
some general results of finite power width flavour were obtained
(see \cite{KG} for the history of the problem and more references),
and \cite{Pu2} for some results for matrices over noncommutative
rings, in particular, for central simple algebras.

\subsection{Monsters} \label{monster}
It was unknown for a long time whether there exists a simple group
of commutator width greater than 1. The first counter-example
\cite{BG} appeared in the context of symplectic geometry and gave a
simple (infinitely generated) group of infinite commutator width.
Later on, using various contexts, there were constructed simple
groups of infinite commutator width which are finitely generated
(\cite{Mu1}, using small cancellation theory) or finitely presented
(\cite{CF}, among quotients of Kac--Moody groups); among groups
appearing in \cite{Mu1}, there are those of arbitrary large finite
commutator width. In \cite{Mu2}, there were constructed simple
groups of infinite square width (and hence infinite commutator
width, because of the equality $[x,y]=x^2(x^{-1}y)^2(y^{-1})^2$,
showing that finite commutator width implies finite square width).
We refer the interested reader to \cite{GG} for a walk along the zoo
of monsters with such anti-Ore and anti-Waring behaviour (and some
conceptual explanations of such phenomena), which grew up in
differential-geometric environment, following the spirit of
\cite{BG}; they are fed up using some advanced techniques, including
quasi-morphisms and quantum cohomology.

It is an interesting question whether such monsters exist among Lie
algebras. The case of finite-dimensional algebras of Cartan type
over fields of positive characteristic was already mentioned above.
As to infinite-dimensional algebras in characteristic zero,
discussions with E.~Zelmanov give a hint that there are infinitely
generated simple Lie algebras of infinite width and finitely
generated ones of arbitrary large finite width. Elaborating such
examples could give a clue to understanding the situation with
associative algebras.

As a final remark, we should mention that apart from Lie algebras,
little is known on the problems discussed in this paper in the case
of non-associative algebras; see, however, \cite{Gordo} and
\cite{Pu1} for some particular results of this flavour. In our
opinion, it would be interesting to study the image of more general
polynomial maps for some classical examples, such as Cayley
octonions, simple quadratic Jordan algebras, simple exceptional
Jordan algebras $HC_3$, and the like.

\noindent {\it Acknowledgements}. This research was supported by the
Israel Science Foundation, grant 1207/12. Kunyavski\u\i \ and
Plotkin were supported in part by the Minerva Foundation through the
Emmy Noether Research Institute of Mathematics. Some questions
raised in this paper were formulated after discussions with
N.~L.~Gordeev, L.~H.~Rowen, A.~Shalev, N.~A.~Vavilov and E.~I.
Zelmanov to whom we are heartily grateful.


\begin{thebibliography}{LOST2}

\bibitem[Am1]{Am1}
S. A. Amitsur, {\it On rings with identities}, J. London Math. Soc.
20 (1955), 464--470.

\bibitem[Am2]{Am2}
S. A. Amitsur, {\it The T-ideals of the free ring}, J. London Math.
Soc. 20 (1955), 470--475.

\bibitem[AR]{AR}
S. A. Amitsur, L. H. Rowen, {\it Elements of reduced trace 0},
Israel J. Math. 87 (1994), 161--179.

\bibitem[ABGV]{ABGV}
A. Auel, E. Brussel, S. Garibaldi, U. Vishne, {\it Open problems on
central simple algebras}, Transform. Groups 16 (2011), 219--264.

\bibitem[BGG]{BGG}
T. Bandman, S. Garion, F. Grunewald, {\it On the surjectivity of
Engel words on $\PSL(2,q)$}, Groups Geom. Dyn. 6 (2012), 409--439.

\bibitem[BGK]{BGK}
T.~Bandman, S. Garion, B.~Kunyavski\u\i , {\it Equations in simple
matrix groups: algebra, geometry, arithmetic, dynamics}, Central
Europ. J. Math., to appear.

\bibitem[BGKP]{BGKP}
T.~Bandman, N. Gordeev, B.~Kunyavski\u\i , E.~Plotkin, {\it
Equations in simple Lie algebras}, J. Algebra 355 (2012), 67--79.

\bibitem[BG]{BG}
J. Barge, E. Ghys, {\it Cocycles d'Euler et de Maslov}, Math. Ann.
294 (1992), 235--265.

\bibitem[BM]{BM}
A. Basmajan, B. Maskit, {\it Space form isometries as commutators
and products of involutions}, Trans. Amer. Math. Soc. 364 (2012),
5015--5033.


\bibitem[Be]{Be}
A. Ya. Belov, {\it No associative PI-algebra coincides with its
commutant}, Sibirsk. Mat. Zh. 44 (2003), 1239--1254; English transl.
in Siberian Math. J. 44 (2003), 969--980.

\bibitem[Bo]{Bo}
A. Borel, {\it On free subgroups of semisimple groups}, Enseign.
Math. 29 (1983), 151--164; reproduced in {\OE}uvres - Collected
Papers, vol.~IV, Springer-Verlag, Berlin--Heidelberg, 2001,
pp.~41--54.

\bibitem[Br]{Br}
G. Brown, {\it On commutators in a simple Lie algebra}, Proc. Amer.
Math. Soc. 14 (1963), 763--767.

\bibitem[CF]{CF}
P.-E. Caprace, K. Fujiwara, {\it Rank-one isometries of buildings
and quasi-morphisms of Kac--Moody groups}, Geom. Funct. Anal. 19
(2010), 1296--1319.

\bibitem[CRK]{CRK}
M. Casals-Ruiz, I. Kazachkov, {\it On systems of equations over free
partially commutative groups}, Mem. Amer. Math. Soc. 212 (2011),
no.~999.

\bibitem[Cha1]{Cha1}
P. Chatterjee, {\it On the surjectivity of the power maps of
algebraic groups in characteristic zero}, Math. Res. Lett. 9 (2002),
741--756.

\bibitem[Cha2]{Cha2}
P. Chatterjee, {\it On the surjectivity of the power maps of
semisimple algebraic groups}, Math. Res. Lett. 10 (2003), 625--633.

\bibitem[Cha3]{Cha3}
P. Chatterjee, {\it On the power maps, orders and exponentiality of
$p$-adic algebraic groups}, J. reine angew. Math. 629 (2009),
201--220.

\bibitem[Cha4]{Cha4}
P. Chatterjee, {\it Surjectivity of power maps of real algebraic
groups}, Adv. Math. 226 (2011), 4639--4666.

\bibitem[Chu]{Chu}
C.-L. Chuang, {\it On ranges of polynomials in finite matrix rings},
Proc. Amer. Math. Soc. 110 (1990), 293--302.

\bibitem[DS]{DS}
P. Deligne, D. Sullivan, {\it Division algebras and the
Hausdorff--Banach--Tarski paradox}, Enseign. Math. 29 (1983),
145--150.

\bibitem[DV]{DV}
R. K. Dennis, L. N. Vaserstein, {\it Commutators in linear groups},
K-Theory 2 (1989), 761--767.

\bibitem[Do]{Do}
S. Donkin, {\it Invariants of several matrices}, Invent. Math. 110
(1992), 389--401.

\bibitem[DR]{DR}
M. Droste, I. Rivin, {\it On extension of coverings}, Bull. Lond.
Math. Soc. 42 (2010), 1044--1054.

\bibitem[DT]{DT}
M. Droste, J. K. Truss, {\it On representing words in the
automorphism group of the random graph}, J. Group Theory 9 (2006),
815--836.

\bibitem[ET]{ET}
A. Elkasapy, A. Thom, {\it About Got\^o's method showing
surjectivity of word maps}, \url{arXiv:1207.5596}.

\bibitem[EG]{EG}
E.~W.~Ellers, N.~Gordeev, {\it On the conjectures of J.~Thompson and
O.~Ore}, Trans. Amer. Math. Soc. 350 (1998), 3657--3671.

\bibitem[Fo1]{Fo1}
E. Formanek, {\it Central polynomials for matrix rings}, J. Algebra
23 (1972), 129--132.

\bibitem[Fo2]{Fo2}
E. Formanek, {\it The ring of generic matrices}, J. Algebra 258
(2002), 310--320.

\bibitem[GG]{GG}
J.-M. Gambaudo, E. Ghys, {\it Commutators and diffeomorphisms of
surfaces}, Ergodic Theory Dynam. Systems 24 (2004), 1591--1617.

\bibitem[Gorde]{Gorde}
N. Gordeev, {\it Sums of orbits of algebraic groups}, I, J. Algebra
295 (2006), 62--80.

\bibitem[Gordo]{Gordo}
S. R. Gordon, {\it Associators in simple algebras}, Pacific J. Math.
51 (1974), 131--141.

\bibitem[Gow]{Gow}
R. Gow, {\it Commutators in finite simple groups of Lie type}, Bull.
London Math. Soc. 32 (2000), 311--315.

\bibitem[Gr]{Gr}
A. B. Gray, Jr.,  {\it Infinite symmetric groups and monomial
groups}, Ph.D. Thesis, New Mexico State Univ., 1960.

\bibitem[GH]{GH}
C. K. Gupta, W. Ho{\l}ubowski, {\it Commutator subgroup of
Vershik--Kerov group}, Linear Algebra Appl. 436 (2012), 4279--4284.

\bibitem[GK]{GK}
R. M. Guralnick, W. M. Kantor, {\it The probability of generating a
simple group}, J. Algebra 234 (2000), 743--792.

\bibitem[GM]{GM}
R. Guralnick, G. Malle, {\it Products of conjugacy classes and fixed
point spaces}, J. Amer. Math. Soc. 25 (2012), 77--121.

\bibitem[HS]{HS}
P. de la Harpe, G. Skandalis, {\it Sur la simplicit\'e essentielle
du groupe des inversibles et du groupe unitaire dans une
$C^*$-alg\`ebre simple}, J. Funct. Anal. 62 (1985), 354--378.

\bibitem[HSVZ1]{HSVZ1}
R. Hazrat, A. Stepanov, N. Vavilov, Z. Zhang, {\it The yoga of
commutators}, Zap. Nauch. Sem. POMI 387 (2011), 53--82; English
transl. in J. Math. Sci. (New York) 179 (2011), 662--678.

\bibitem[HSVZ2]{HSVZ2}
R. Hazrat, A. Stepanov, N. Vavilov, Z. Zhang, {\it Commutator width
in Chevalley groups}, \url{arXiv:1206.2128}.

\bibitem[Is]{Is}
I. M. Isaacs, {\it Commutators and the commutator subgroup}, Amer.
Math. Monthly 84 (1977), 720--722.

\bibitem[JLO]{JLO}
S. Jambor, M. W. Liebeck, E. A. O${}^{\prime}$Brien, {\it Some word
maps that are non-surjective on infinitely many finite simple
groups}, \url{arXiv:1205.1952}.

\bibitem[KBMR]{KBMR}
A. Kanel-Belov, S. Malev, L. Rowen, {\it The images of
non-commutative polynomials evaluated on $2\times 2$ matrices},
Proc. Amer. Math. Soc. 140 (2012), 465--478.

\bibitem[KN]{KN}
M. Kassabov, N. Nikolov, {\it Words with few values in finite simple
groups}, Quart. J. Math., to appear.

\bibitem[KG]{KG}
S. A. Katre, A. S. Garge, {\it Matrices over commutative rings as
sums of $k$-th powers}, Proc. Amer. Math. Soc. 141 (2013), 103--113.

\bibitem[Kl]{Kl} E. Klimenko, Engel maps for $\PSL_2(\CC)$, Preprint.

\bibitem[Ku]{Ku}
B. Kunyavskii, {\it The Bogomolov multiplier of finite simple
groups}, in: ``Cohomological and Geometric Approaches to Rationality
Problems'' (F.~Bogomolov, Yu.~Tschinkel, Eds.), Progr. Math.,
vol.~282, Birkh\"auser, Boston, MA, 2010, pp.~209--217.

\bibitem[La]{La}
M. Larsen, {\it Word maps have large image}, Israel J. Math. 139
(2004), 149--156.

\bibitem[LST1]{LST1}
M. Larsen, A. Shalev, P. H. Tiep, {\it Waring problem for finite
simple groups}, Ann. Math. 174 (2011), 1885--1950.

\bibitem[LST2]{LST2}
M. Larsen, A. Shalev, P. H. Tiep, {\it Waring problem for finite
quasisimple groups},  Intern. Math. Res. Notices, to appear.

\bibitem[LB]{LB}
L. Le Bruyn, {\it Centers of generic division algebras, the
rationality problem 1965--1990}, Israel J. Math. 76 (1991), 97--111.

\bibitem[Le]{Le}
M. Levy, {\it Word maps with small image in simple groups},
\url{arXiv:1206.1206}.

\bibitem[LOST1]{LOST1}
M. W. Liebeck, E. A. O${}^{\prime}$Brien, A. Shalev, P. H. Tiep,
{\it The Ore conjecture}, J. Europ. Math. Soc. 12 (2010), 939--1008.

\bibitem[LOST2]{LOST2}
M. W. Liebeck, E. A. O${}^{\prime}$Brien, A. Shalev, P. H. Tiep,
{\it Commutators in finite quasisimple groups}, Bull. London Math.
Soc. 43 (2011), 1079--1092.

\bibitem[LS]{LS}
M. W. Liebeck, A. Shalev, {\it Diameters of finite simple groups:
sharp bounds and applications}, Ann. Math. 154 (2001), 383--406.

\bibitem[Ly]{Ly}
R. C. Lyndon, {\it Words and infinite permutations}, in: ``Mots.
Langue, Raisonnement, Calcul'',  Herm\`es, Paris, 1990,
pp.~143--152.

\bibitem[ML]{ML}
L. Makar-Limanov, {\it Algebraically closed skew fields}, J. Algebra
93 (1985), 117--135.

\bibitem[Ma]{Ma}
J. A. Maroli, {\it Representation of tree permutations by words},
Proc. Amer. Math. Soc. 110 (1990), 859--869.

\bibitem[MZ]{MZ}
C. Martinez, E. Zelmanov, {\it Products of powers in finite simple
groups}, Israel J. Math. 96 (1996), 469--479.

\bibitem[Mer]{Mer}
Ju. I. Merzljakov, {\it Algebraic linear groups as full groups of
automorphisms and the closure of their verbal subgroups}, Algebra i
Logika Sem. 6 (1967), no.~1, 83--94. (Russian.)

\bibitem[Mes]{Mes}
Z. Mesyan, {\it Commutator rings}, Bull. Austral. Math. Soc. 74
(2006), 279--288.

\bibitem[MP]{MP}
J. Morita, E. Plotkin, {\it Gauss decompositions of Kac--Moody
groups}, Comm. Algebra 27 (1999), 465--475.

\bibitem[Mu1]{Mu1}
A. Muranov, {\it Finitely generated infinite simple groups of
infinite commutator width}, Internat. J. Algebra Computation 17
(2007), 607--659.

\bibitem[Mu2]{Mu2}
A. Muranov, {\it Finitely generated infinite simple groups of
infinite square width and vanishing stable commutator length}, J.
Topol. Anal. 2 (2010), 341--384.

\bibitem[My]{My}
J. Mycielski, {\it Can one solve equations in groups?}, Amer. Math.
Monthly 84 (1977), 723--726.

\bibitem[Ni]{Ni}
N. Nikolov, {\it Algebraic properties of profinite groups},
\url{arXiv:1108.5130}.

\bibitem[NS]{NS}
N. Nikolov, D. Segal, {\it Powers in finite groups}, Groups Geom.
Dyn. 5 (2011), 501--507.

\bibitem[Or]{Or}
O. Ore, {\it Some remarks on commutators}, Proc. Amer. Math. Soc. 2
(1951), 307--314.

\bibitem[Pr]{Pr}
C. Procesi, {\it The invariant theory of $n\times n$ matrices}, Adv.
Math. 19 (1976), 306--381.

\bibitem[Pu1]{Pu1}
S. Pumpl\"un, {\it Sums of squares in octonion algebras}, Proc.
Amer. Math. Soc. 133 (2005), 3143--3152.

\bibitem[Pu2]{Pu2}
S. Pumpl\"un, {\it Sums of $d$-th powers in non-commutative rings},
Beitr\"age Algebra Geom. 48 (2007), 291--301.

\bibitem[Ra]{Ra}
Yu. Razmyslov, {\it On a problem of Kaplansky}, Izv. Akad. Nauk
SSSR. Ser. Mat. 73 (1973), 483--501; English transl. in: Math USSR
Izv. 7 (1973), 479--496.

\bibitem[Ree]{Ree}
R. Ree, {\it Commutators in semi-simple algebraic groups}, Proc.
Amer. Math. Soc. 15 (1964), 457--460.

\bibitem[Ros]{Ros}
M. Rosset, {\it Elements of trace zero and commutators}, Ph.D.
Thesis, Bar-Ilan Univ., 1997.

\bibitem[RR1]{RR1}
M. Rosset, S. Rosset, {\it Elements of trace zero in central simple
algebras}, in: ``Rings, Extensions, and Cohomology (Evanston, IL,
1993)'' (A.~R.~Magid, Ed.), Lecture Notes Pure Appl. Math.,
vol.~159, Marcel Dekker, New York, 1994, pp.~205--211.

\bibitem[RR2]{RR2}
M. Rosset, S. Rosset, {\it Elements of trace zero that are not
commutators}, Comm. Algebra 28 (2000), 3059--3072.

\bibitem[Row]{Row}
L. H. Rowen, {\it Polynomial identities in ring theory}, Academic
Press, New York, 1980.

\bibitem[SW]{SW}
J. Saxl, J. S. Wilson, {\it A note on powers in simple groups},
Math. Proc. Cambridge Phil. Soc. 122 (1997), 91--94.

\bibitem[Se]{Se}
D. Segal, {\it Words: notes on verbal width in groups}, London Math.
Soc. Lecture Notes Ser., vol.~361, Cambridge Univ. Press, Cambridge,
2009.

\bibitem[Sh1]{Sh1}
A. Shalev, {\it Commutators, words, conjugacy classes and character
methods}, Turkish J. Math. 31 (2007), 131--148.

\bibitem[Sh2]{Sh2}
A. Shalev, {\it Applications of some zeta functions in group
theory}, in: ``Zeta Functions in Algebra and Geometry'' (A.~Campillo
{\it et al.}, Eds.), Contemp. Math., vol.~566, Amer. Math. Soc.,
Providence, RI, 2012, pp.~331--344.

\bibitem [Stei]{Stei}
A. Stein, {\it $1\frac{1}{2}$-generation of finite simple groups},
Beitr\"age Algebra Geom. 39 (1998), 349--358.

\bibitem[Stb]{Stb}
R. Steinberg, {\it On power maps in algebraic groups}, Math. Res.
Lett. 10 (2003), 621--624.

\bibitem [Step]{Step}
A. Stepanov, {\it Universal localization in algebraic groups},
preprint, 2010, available at
\url{http://alexei.stepanov.spb.ru/papers/formal.pdf}.

\bibitem[Th]{Th}
A. Thom, {\it Convergent sequences in discrete groups}, Canad. Math.
Bull. 56 (2013), 424--433.

\bibitem[Va]{Va}
L. N. Vaserstein, {\it On the sum of powers of matrices}, Linear
Multilin. Algebra 21 (1987), 261--270.

\bibitem[VW]{VW}
L. N. Vaserstein, E. Wheland, {\it Commutators and companion
matrices over rings of stable rank $1$}, Linear Algebra Appl. 142
(1990), 263--277.

\bibitem[YCW]{YCW}
S.-K. Ye, S. Chen, C.-S. Wang, {\it Gauss decomposition with
prescribed semisimple part in quadratic groups}, Comm. Algebra 37
(2009), 3054--3063.

\end{thebibliography}
\end{document}